\documentclass[12pt]{amsart}

\usepackage{amsmath}
\usepackage{amscd}
\usepackage{amssymb}
\usepackage{amsfonts}
\usepackage{hyperref}

\newtheorem{theorem}{Theorem}[section]
\newtheorem{lemma}[theorem]{Lemma}
\newtheorem{corollary}[theorem]{Corollary}
\newtheorem{proposition}[theorem]{Proposition}

\theoremstyle{definition}

\theoremstyle{remark}
\newtheorem{remark}[theorem]{Remark}
\numberwithin{equation}{section}

\begin{document}

\title[K\"ahler--Ricci solitons with CSC]
{Rigidity of K\"ahler--Ricci Solitons with Constant Scalar Curvature}
\author[M. Hor\'acio]{Matheus  Hor\'acio}
\address{Departamento de Matem\'atica, Universidade de Bras\'ilia,
70910-900 Bras\'ilia, DF, Brazil}
\email{matheus.andrade5488@gmail.com}
\thanks{M. Hor\'acio was financed by the Coordena\c{c}\~ao de
Aperfei\c{c}oamento de Pessoal de N\'ivel Superior -- Brasil (CAPES) --
Finance Code 001.}
\author[W. Li]{Wenqi Li}
\address{School of Mathematical Sciences, Key Laboratory of MEA (Ministry of
Education) and Shanghai Key Laboratory of PMMP, East China Normal University,
Shanghai 200241, China}
\email{51255500054@stu.ecnu.edu.cn}

\author[J. Ou]{Jianyu Ou}
\address{School of Mathematical Sciences, Xiamen University,
Xiamen 361005, China}
\email{oujianyu@xmu.edu.cn}
\thanks{Jianyu Ou is partially supported by NSFC grant No.12401074, 12371061,
12371081, and Fundamental Research Funds for the Central Universities no.
20720250006.}

\author[G. Wu]{Guoqiang Wu}
\address{School of Science, Zhejiang Sci-Tech University,
Hangzhou 310018, China}
\email{gqwu@zstu.edu.cn}

\author[D. Zhou]{Detang Zhou}
\address{Instituto de Matem\'atica e Estat\'istica,
Universidade Federal Fluminense, S\~ao Domingos,
Niter\'oi, RJ 24210-201, Brazil}
\email{zhoud@id.uff.br}
\thanks{D. Zhou was partially supported by CNPq/Brazil 308067/2023-1.}

\date{}

\begin{abstract}
Let $(M^{2m},g,f,J)$ be a complete nonsteady gradient
K\"ahler--Ricci soliton satisfying
\(
\operatorname{Ric}+\nabla^2f=\lambda g,
\qquad \lambda\neq0.
\)
We prove that constant scalar curvature forces the soliton to be rigid.
More precisely, the universal cover splits holomorphically and
isometrically as
\(
N^{2k}\times\mathbb C^{m-k},
\)
where $N^{2k}$ is K\"ahler--Einstein with
$\operatorname{Ric}_{g_N}=\lambda g_N$. In the shrinking case the
quotient is trivial; in the normalization $\lambda=\frac12$ one has
$R\equiv k$.

The K\"ahler result rests on a Riemannian rigidity criterion. On a
complete nonsteady gradient Ricci soliton, if
$\mathcal L_{\nabla f}\operatorname{Ric}$ is nonnegative or nonpositive
everywhere, then it vanishes and the soliton is rigid; no assumption on
the scalar curvature is needed. Along the Ricci flow generated by the
soliton, this means that a Ricci tensor that is monotone in time is
constant in time, and that this forces rigidity.

We also give a direct proof that the pinching
\(
0\leq\operatorname{Ric}\leq\lambda g
\)
forces constant scalar curvature and radial flatness, yielding the
rigidity conclusion through the Petersen--Wylie characterization. The
proof combines a weighted cutoff argument with a partial Codazzi
symmetry for the Ricci endomorphism.
\end{abstract}

\maketitle

\section{Introduction}

A gradient Ricci soliton is a triple $(M^n,g,f)$ satisfying
\[
\operatorname{Ric}+\nabla^2f=\lambda g
\]
for some constant $\lambda$. It is shrinking, steady, or expanding
according as $\lambda>0$, $\lambda=0$, or $\lambda<0$. A gradient
K\"ahler--Ricci soliton is a gradient Ricci soliton
$(M^{2m},g,f,J)$ for which $(M,g,J)$ is K\"ahler and $\nabla f$ is a
holomorphic vector field.
Throughout the paper, all manifolds are assumed to be connected unless otherwise stated.

Following Petersen and Wylie \cite{Petersen-Wylie2}, a complete
nonsteady gradient Ricci soliton is called \emph{rigid} if it is a flat
vector bundle whose universal cover is a product
\[
N\times\mathbb R^\ell,
\]
where $N$ is Einstein with Einstein constant $\lambda$, and the
potential is Gaussian on the Euclidean factor. Petersen and Wylie also
proved that rigidity is characterized by constant scalar curvature and
radial flatness \cite[Theorem~1.2]{Petersen-Wylie2}.

For complete shrinkers, Cao conjectured that constant scalar curvature
alone should imply rigidity. Rigidity under additional curvature
hypotheses was obtained in several works; see, for example,
Fern\'{a}ndez-L\'{o}pez--Garc\'{i}a-R\'{i}o \cite{FR10,FR16},
Munteanu--Sesum \cite{Munteanu-Sesum}, and
Catino--Mastrolia--Monticelli \cite{CMM17}. Cheng and Zhou
\cite{Cheng-Zhou} proved the conjecture in dimension four, and further
progress appears in \cite{Ou-Qu-Wu,Li-Ou-Qu-Wu,Li-Ou-Qu-Wu3}.

In the K\"ahler setting, Fern\'{a}ndez-L\'{o}pez and
Garc\'{i}a-R\'{i}o \cite{FR16} obtained low-dimensional rigidity
results under constant scalar curvature. Chen and Zhu \cite{Chen-Zhu}
proved rigidity under the harmonic Bochner tensor condition, and Zhang
\cite{Zhang2022} obtained higher-dimensional results under additional
divergence assumptions. Complete K\"ahler--Ricci shrinker surfaces have
also been classified through the work of Conlon--Deruelle--Sun
\cite{CDS}, Bamler--Cifarelli--Conlon--Deruelle \cite{BCCD}, and
Li--Wang \cite{Li-Wang}.Recently, Sun-Zhang \cite{Sun-Zhang2024} proved that
all complete gradient K\"ahler Ricci shrinkers are quasi-projective.

Our first main result is the following.

\begin{theorem}\label{mainthm}
Let $(M^{2m},g,f,J)$ be a complete connected nonsteady gradient
K\"ahler--Ricci soliton satisfying
\[
\operatorname{Ric}+\nabla^2f=\lambda g,
\qquad \lambda\neq0.
\]
Assume that the scalar curvature $R$ is constant. Then there exists
$k\in\{0,\ldots,m\}$ such that
\[
R=2k\lambda,
\]
and the universal cover is holomorphically and isometrically a product
\[
(\widetilde M,\widetilde g,\widetilde J)
\cong
\bigl(N^{2k}\times\mathbb C^{m-k},
 g_N+g_{\mathrm{Euc}},J_N\oplus J_0\bigr),
\]
where $(N^{2k},g_N,J_N)$ is a complete simply connected
K\"ahler--Einstein manifold satisfying
\[
\operatorname{Ric}_{g_N}=\lambda g_N.
\]
Consequently, $M$ is a quotient of this product by a discrete group of
holomorphic isometries acting freely and properly discontinuously.
If $\lambda>0$, then the quotient is trivial. In particular, under the
normalization $\lambda=\frac12$,
\[
R=k,
\qquad
M^{2m}\cong N^{2k}\times\mathbb C^{m-k}.
\]
\end{theorem}

Since every complete steady gradient Ricci soliton with constant scalar
curvature is Ricci-flat, combining this fact with Theorem~\ref{mainthm}
yields the following.

\begin{corollary}\label{cor:csc-rigid}
Let $(M^{2m},g,f,J)$ be a complete connected gradient K\"ahler--Ricci
soliton with constant scalar curvature. Then the soliton is rigid.
\end{corollary}

The proof of Theorem~\ref{mainthm} is based on the identity
\[
\mathcal L_{\nabla f}\operatorname{Ric}=0,
\]
which follows in the K\"ahler setting from constant scalar curvature,
the closedness of the Ricci form, and the real holomorphicity of
$\nabla f$. A purely Riemannian argument then yields
$A^2=\lambda A$ for the Ricci endomorphism $A$ and proves radial flatness directly.
This Riemannian step holds in greater generality.

\begin{theorem}\label{thm:criterion-main}
Let $(M^n,g,f)$ be a complete nonsteady gradient Ricci soliton
satisfying
\[
\operatorname{Ric}+\nabla^2f=\lambda g,
\qquad \lambda\neq0.
\]
If
\[
\mathcal L_{\nabla f}\operatorname{Ric}\geq0
\quad\text{on }M
\qquad\text{or}\qquad
\mathcal L_{\nabla f}\operatorname{Ric}\leq0
\quad\text{on }M,
\]
then $\mathcal L_{\nabla f}\operatorname{Ric}=0$ the scalar curvature
is constant, and the soliton is rigid.
\end{theorem}

Thus, the constancy of the scalar curvature in
Theorem~\ref{thm:criterion-main} is part of the conclusion rather than
a hypothesis. A simple observation used in its proof may also be useful
in other questions concerning gradient Ricci solitons: on any complete
nonsteady gradient Ricci soliton,
\[
\nabla_{\nabla f}R=0
\quad\Longrightarrow\quad
R\ \text{is constant}.
\]
No curvature sign condition or K\"ahler assumption is required.  

The criterion has a natural interpretation in terms of
the Ricci flow. The soliton generates the self-similar solution
$g(t)=(1-2\lambda t)\Phi_{\sigma(t)}^*g$ of
$\partial_tg=-2\operatorname{Ric}_{g(t)}$, where $\Phi_t$ is the flow
of $\nabla f$ and $\sigma$ is a reparametrization of time, and
$\operatorname{Ric}_{g(t)}=\Phi_{\sigma(t)}^*\operatorname{Ric}$. Hence
$\partial_t\operatorname{Ric}_{g(t)}$ is, up to a positive factor, the
pullback of $\mathcal L_{\nabla f}\operatorname{Ric}$, which on a
soliton equals $\Delta_L\operatorname{Ric}$.
Theorem~\ref{thm:criterion-main} thus says that if the Ricci tensor of
this flow is monotone in time, then it is constant in time, and the
soliton is rigid; see Theorem~\ref{thm:riemannian-rigidity} and
Remark~\ref{rem:stationarity}. In the K\"ahler case the sign condition
holds whenever the scalar curvature is plurisubharmonic or
plurisuperharmonic (Corollary~\ref{cor:psh}).

We also give a self-contained proof of the following pinching result.

\begin{theorem}\label{thm:pinching-main}
Let $(M^n,g,f)$ be a complete connected gradient shrinking Ricci soliton
satisfying
\[
\operatorname{Ric}+\nabla^2f=\lambda g,
\qquad \lambda>0.
\]
If
\[
0\leq\operatorname{Ric}\leq\lambda g,
\]
then the scalar curvature is constant,
\[
A(\lambda I-A)=0,
\]
and
\[
R(X,\nabla f)\nabla f=0
\]
for every vector field $X$. Consequently, the soliton is rigid.
\end{theorem}

Theorem~\ref{thm:pinching-main} is due to Naber \cite[Theorem~2]{Naber}
and to Petersen and Wylie
\cite[Lemma~1.4 and Proposition~1.3(2)]{Petersen-Wylie2}. Our proof is
self-contained up to the standard characterization of rigidity by
constant scalar curvature and radial flatness. The constancy of $R$
follows from a weighted cutoff argument. Once $A^2=\lambda A$ is known,
the Ricci Bochner identity forces a partial Codazzi symmetry for $A$,
and this symmetry gives radial flatness.

The partial Codazzi symmetry is related to recent rigidity results for
Codazzi operators. Cheng, Lipa and Zhou \cite{Cheng-Lipa-Zhou} use
scalar curvature information together with trace invariants of a Codazzi
symmetric tensor to obtain rigidity of its eigenvalues. In the present
paper, the Ricci endomorphism is not assumed to be a Codazzi tensor.
Instead, the relation $A^2=\lambda A$ and the Ricci Bochner identity
force the restricted symmetry
\[
(\nabla_XA)Y=(\nabla_YA)X
\]
when $X$ and $Y$ belong to the same Ricci eigendistribution. This gives
a useful connection between the two rigidity mechanisms.

The paper is organized as follows. Section~2 collects the soliton
identities and curvature conventions used throughout. Section~3 proves
the K\"ahler identity $\mathcal L_{\nabla f}\operatorname{Ric}=0$.
Section~4 proves Theorem~\ref{thm:criterion-main}, in the form of
Theorem~\ref{thm:riemannian-rigidity}, and Theorem~\ref{mainthm}. Section~5 proves
Theorem~\ref{thm:pinching-main}, discusses the partial Codazzi symmetry.

\section{Preliminaries}

Let $(M^n,g,f)$ be a gradient Ricci soliton
\begin{equation}\label{eq:soliton}
\operatorname{Ric}+\nabla^2f=\lambda g.
\end{equation}
Set
\[
V=\nabla f
\]
and let $A:TM\to TM$ denote the Ricci endomorphism,
\[
\operatorname{Ric}(X,Y)=g(AX,Y).
\]
We use
\[
\Delta_f=\Delta-\langle\nabla f,\nabla\cdot\rangle.
\]

Our curvature convention is
\[
R(X,Y)Z
=
\nabla_X\nabla_YZ-\nabla_Y\nabla_XZ-\nabla_{[X,Y]}Z,
\]
and for orthonormal $X,Y$ we write
\[
K(X,Y)=\langle R(X,Y)Y,X\rangle.
\]

We use the standard soliton identities
\begin{align}
\nabla_XV&=(\lambda I-A)X,
\label{eq:hess}\\
\nabla R&=2A(V),
\label{eq:gradR-general}\\
\Delta_fR&=2\lambda R-2|\operatorname{Ric}|^2.
\label{eq:scalar-drift}
\end{align}
Differentiating \eqref{eq:soliton} and commuting derivatives gives
\begin{equation}\label{eq:soliton-commutator}
(\nabla_XA)Y-(\nabla_YA)X
=
-R(X,Y)V.
\end{equation}
The contracted Bianchi identity is
\begin{equation}\label{eq:divA}
\operatorname{div}A=\frac12\nabla R,
\end{equation}
where
\[
\operatorname{div}A=\sum_i(\nabla_{e_i}A)e_i.
\]

For the Ricci tensor, the $f$-Bochner formula is
\begin{equation}\label{eq:bochner-ricci}
\frac12\Delta_f|\operatorname{Ric}|^2
=
|\nabla\operatorname{Ric}|^2
+
2\lambda|\operatorname{Ric}|^2
-
2Q(A),
\end{equation}
where, at a point at which $Ae_i=\mu_i e_i$ in an orthonormal basis,
\begin{equation}\label{eq:QA}
Q(A)
=
\sum_{i,j}\mu_i\mu_jK_{ij},
\qquad
K_{ij}=\langle R(e_i,e_j)e_j,e_i\rangle.
\end{equation}

If $\lambda>0$, a complete gradient Ricci shrinker has finite weighted
volume,
\begin{equation}\label{eq:finite-weighted-volume}
\int_Me^{-f}\,\mathrm{d}v<\infty.
\end{equation}
For the noncompact case this follows, for instance, from the standard
potential and volume estimates of Cao--Zhou \cite{Cao-Zhou}; the compact
case is immediate.

\section{The K\"ahler identity}

The K\"ahler structure gives a useful expression for the Lie derivative
of the Ricci tensor without assuming that the scalar curvature is
constant.

\begin{proposition}\label{prop:kahler-lie}
Let $(M^{2m},g,f,J)$ be a gradient K\"ahler--Ricci soliton satisfying
\[
\operatorname{Ric}+\nabla^2f=\lambda g,
\]
and set $V=\nabla f$. Then
\begin{equation}\label{eq:kahler-rho-lie}
\mathcal L_V\rho=\frac12\,\mathrm d\mathrm d^cR,
\end{equation}
where $\rho(X,Y)=\operatorname{Ric}(JX,Y)$ and
$\mathrm d^cu=-\mathrm du\circ J$. Moreover,
\begin{equation}\label{eq:kahler-lie}
(\mathcal L_V\operatorname{Ric})(X,Y)
=\frac12\bigl(\nabla^2R(X,Y)+\nabla^2R(JX,JY)\bigr).
\end{equation}
\end{proposition}

\begin{proof}
Since $(M,g,J)$ is K\"ahler, the Ricci form is closed,
\[
\mathrm d\rho=0.
\]
By \eqref{eq:gradR-general} and the $J$-invariance of the Ricci tensor,
\begin{align*}
(\iota_V\rho)(Y)
&=\rho(V,Y)
=\operatorname{Ric}(JV,Y)\\
&=-\operatorname{Ric}(V,JY)
=-\frac12\,\mathrm dR(JY)
=\frac12\,\mathrm d^cR(Y).
\end{align*}
Hence Cartan's formula gives
\[
\mathcal L_V\rho
=\mathrm d(\iota_V\rho)+\iota_V\mathrm d\rho
=\frac12\,\mathrm d\mathrm d^cR,
\]
which proves \eqref{eq:kahler-rho-lie}.

We also have
\begin{equation}\label{eq:lie-J}
\mathcal L_VJ=0.
\end{equation}
Indeed, since $\nabla J=0$,
\begin{align*}
(\mathcal L_VJ)(Y)
&=[V,JY]-J[V,Y]\\
&=-\nabla_{JY}V+J\nabla_YV.
\end{align*}
Using \eqref{eq:hess} and the fact that $A$ commutes with $J$,
\[
\nabla_{JY}V=(\lambda I-A)JY
=J(\lambda I-A)Y
=J\nabla_YV,
\]
so \eqref{eq:lie-J} follows.

Since
\[
\rho(X,Y)=\operatorname{Ric}(JX,Y),
\]
we obtain
\[
(\mathcal L_V\rho)(X,Y)
=(\mathcal L_V\operatorname{Ric})(JX,Y)
+\operatorname{Ric}((\mathcal L_VJ)X,Y).
\]
By \eqref{eq:lie-J}, the second term vanishes. On the other hand,
$\nabla J=0$ gives
\[
\mathrm d\mathrm d^cR(X,JY)
=\nabla^2R(X,Y)+\nabla^2R(JX,JY).
\]
Combining this with \eqref{eq:kahler-rho-lie} and replacing $X$ by
$-JX$ yields \eqref{eq:kahler-lie}.
\end{proof}

\begin{theorem}\label{thm:lie-ricci}
Let $(M^{2m},g,f,J)$ be a gradient K\"ahler--Ricci soliton satisfying
\[
\operatorname{Ric}+\nabla^2f=\lambda g.
\]
If the scalar curvature is constant, then
\[
\mathcal L_{\nabla f}\operatorname{Ric}=0.
\]
\end{theorem}

\begin{proof}
If $R$ is constant, then $\nabla^2R=0$. The conclusion follows
immediately from Proposition~\ref{prop:kahler-lie} and
\eqref{eq:kahler-lie}.
\end{proof}

\section{A Riemannian Lie-derivative criterion}

The results in this section are purely Riemannian. On a complete
gradient Ricci soliton the vector field $\nabla f$ is complete and the
scalar curvature is bounded below \cite[Theorem~1.3]{Zhang-ZH}. We
denote by $\Phi_t$ the flow of $\nabla f$.

\begin{lemma}\label{lem:scalar-along-V}
Let $(M^n,g,f)$ be a complete gradient Ricci soliton satisfying
\eqref{eq:soliton} with $\lambda\neq0$, and set $V=\nabla f$. If
$V(R)=0$, then $R$ is constant.
\end{lemma}

\begin{proof}
Let $\Phi_t$ be the flow of $V$ and fix $p\in M$. Consider the integral
curve
\[
\gamma(t)=\Phi_t(p),
\qquad
\gamma'(t)=V(\gamma(t)).
\]
If $p$ is a critical point of $f$, then $\gamma$ is constant and the
argument below applies with $q=p$.
Using \eqref{eq:hess} and \eqref{eq:gradR-general}, we obtain along
$\gamma$
\begin{align*}
\frac{\mathrm{d}}{\mathrm{d}t}|V|^2(\gamma(t))
&=2\langle\nabla_VV,V\rangle(\gamma(t))\\
&=2\lambda|V|^2(\gamma(t))-2\operatorname{Ric}(V,V)(\gamma(t))\\
&=2\lambda|V|^2(\gamma(t))-V(R)(\gamma(t)).
\end{align*}
By hypothesis $V(R)=0$, and hence
\[
\frac{\mathrm{d}}{\mathrm{d}t}|V|^2(\gamma(t))
=2\lambda|V|^2(\gamma(t)).
\]
Therefore
\begin{equation}\label{eq:V-growth-orbit}
|V|(\gamma(t))=e^{\lambda t}|V|(p).
\end{equation}

Suppose first that $\lambda>0$. For $s<t\leq0$,
\begin{align*}
d(\gamma(s),\gamma(t))
&\leq\int_s^t|\gamma'(u)|\,\mathrm{d}u\\
&=|V|(p)\int_s^t e^{\lambda u}\,\mathrm{d}u.
\end{align*}
Since
\[
\int_{-\infty}^0e^{\lambda u}\,\mathrm{d}u<\infty,
\]
the curve $\gamma(t)$ is Cauchy as $t\to-\infty$. By completeness of
$M$, there exists $q\in M$ such that
\[
\gamma(t)\longrightarrow q
\qquad\text{as }t\to-\infty.
\]
Moreover, \eqref{eq:V-growth-orbit} gives
\[
|V|(\gamma(t))\longrightarrow0,
\]
and therefore, by continuity,
\[
V(q)=0.
\]

If $\lambda<0$, the same argument is applied as $t\to+\infty$, since
\[
\int_0^{+\infty}e^{\lambda u}\,\mathrm{d}u<\infty.
\]
Thus, in either case, the orbit through every point $p\in M$ converges
in one time direction to a point $q\in M$ satisfying
\[
V(q)=0.
\]

On the other hand,
\[
\frac{\mathrm{d}}{\mathrm{d}t}R(\gamma(t))
=V(R)(\gamma(t))=0.
\]
Hence $R$ is constant along $\gamma$, and therefore
\[
R(p)=R(q).
\]
Since $V(q)=0$, \eqref{eq:gradR-general} gives
\[
\nabla R(q)=2A(V(q))=0.
\]

Consequently, for every $p\in M$, the value $R(p)$ is attained at a
critical point of $R$. Thus
\[
R(M)\subseteq\operatorname{CritVal}(R),
\]
where $\operatorname{CritVal}(R)$ denotes the set of critical values of
$R$.

By Sard's theorem, $\operatorname{CritVal}(R)$ has Lebesgue measure
zero. Hence $R(M)$ has measure zero. Since $M$ is connected and $R$ is
continuous, $R(M)\subset\mathbb R$ is an interval. The only interval of
measure zero is a single point. Therefore $R$ is constant on $M$.
\end{proof}

\begin{theorem}\label{thm:riemannian-rigidity}
Let $(M^n,g,f)$ be a complete nonsteady gradient Ricci soliton
satisfying
\[
\operatorname{Ric}+\nabla^2f=\lambda g,
\qquad \lambda\neq0.
\]
Assume that either
\[
\mathcal L_{\nabla f}\operatorname{Ric}\geq0
\quad\text{on }M
\qquad\text{or}\qquad
\mathcal L_{\nabla f}\operatorname{Ric}\leq0
\quad\text{on }M.
\]
Then the flow of $\nabla f$ preserves the Ricci tensor,
\begin{equation}\label{eq:ric-stationary}
\Phi_t^*\operatorname{Ric}=\operatorname{Ric}
\qquad\text{for all }t\in\mathbb R,
\end{equation}
that is, $\mathcal L_{\nabla f}\operatorname{Ric}=0$. Moreover, the
scalar curvature is constant,
\begin{equation}\label{eq:A-polynomial}
A(A-\lambda I)=0,
\end{equation}
and
\[
R(X,\nabla f)\nabla f=0
\]
for every vector field $X$. In particular, the soliton is radially flat
and hence rigid.
\end{theorem}

\begin{proof}
Set $V=\nabla f$ and $T=\mathcal L_V\operatorname{Ric}$. Since
$\frac{\mathrm{d}}{\mathrm{d}t}\Phi_t^*\operatorname{Ric}=\Phi_t^*T$,
\eqref{eq:ric-stationary} is equivalent to $T=0$. For vector fields
$X,Y$, by \eqref{eq:hess},
\begin{equation}\label{eq:lie-nablaV}
\begin{aligned}
T(X,Y)
&=(\nabla_V\operatorname{Ric})(X,Y)
+\operatorname{Ric}(\nabla_XV,Y)
+\operatorname{Ric}(X,\nabla_YV)\\
&=g\bigl((\nabla_VA+2\lambda A-2A^2)X,Y\bigr).
\end{aligned}
\end{equation}
Taking the trace and using \eqref{eq:scalar-drift},
\begin{equation}\label{eq:lie-trace}
\operatorname{tr}_gT
=V(R)+2\lambda R-2|\operatorname{Ric}|^2
=\Delta R.
\end{equation}
Moreover, \eqref{eq:gradR-general} gives
$\operatorname{Ric}(V,V)=\frac12V(R)$. Since $[V,V]=0$,
\begin{equation}\label{eq:lie-second}
T(V,V)=\frac12V(V(R)).
\end{equation}
Fix $p\in M$ and set $\gamma(t)=\Phi_t(p)$, $t\in\mathbb R$. By
\eqref{eq:lie-second}, the second derivative of $R\circ\gamma$ is
$2T(V,V)\circ\gamma$.

Suppose first that $T\leq0$. Then $R\circ\gamma$ is concave on
$\mathbb R$ and bounded below, hence constant.

Suppose now that $T\geq0$. Let $U$ be a parallel unit vector field
along $\gamma$ and set $a(t)=g(AU,U)(\gamma(t))$. By
\eqref{eq:lie-nablaV},
\[
a'=T(U,U)+2|AU|^2-2\lambda a.
\]
Since $A$ is self-adjoint and $|U|=1$, the Cauchy--Schwarz inequality
gives
\[
|AU|^2\geq\langle AU,U\rangle^2=a^2.
\]
Therefore
\[
a'\geq2a(a-\lambda).
\]
Write $\beta=\max\{0,\lambda\}$. If $a(t_0)>\beta$, then
$b=a-\beta$ satisfies $b'\geq2(b+\beta)(b+\beta-\lambda)\geq2b^2$
while $b>0$, so $b$ remains positive on $[t_0,\infty)$ and
$(1/b)'\leq-2$ there, which is impossible. Since the point
$p$ and the initial unit vector are arbitrary, it follows that
\[
\operatorname{Ric}\leq\beta g
\]
on $M$. Hence $R\leq n\beta$, and $R\circ\gamma$ is convex on
$\mathbb R$ and bounded above, hence constant.

In both cases $V(R)=0$, and Lemma~\ref{lem:scalar-along-V} shows that
$R$ is constant. Then \eqref{eq:lie-trace} gives
$\operatorname{tr}_gT=0$, and the semidefinite sign of $T$ gives $T=0$.
Thus \eqref{eq:ric-stationary} holds, and \eqref{eq:lie-nablaV} reduces
to
\begin{equation}\label{eq:nablaVA}
\nabla_VA
=
2(A^2-\lambda A)
=
2A(A-\lambda I).
\end{equation}

Since $R$ is constant, \eqref{eq:scalar-drift} gives
\[
\operatorname{tr}(A^2)
=
|\operatorname{Ric}|^2
=
\lambda R.
\]
Thus $\operatorname{tr}(A^2)$ is constant. Differentiating in the
$V$-direction and using \eqref{eq:nablaVA},
\[
0
=
2\operatorname{tr}(A\nabla_VA)
=
4\left(\operatorname{tr}(A^3)-\lambda\operatorname{tr}(A^2)\right).
\]
Hence
\begin{equation}\label{eq:p3}
\operatorname{tr}(A^3)=\lambda\operatorname{tr}(A^2).
\end{equation}
The right-hand side is constant, so $\operatorname{tr}(A^3)$ is
constant. Differentiating once more,
\[
0
=
3\operatorname{tr}(A^2\nabla_VA)
=
6\left(\operatorname{tr}(A^4)-\lambda\operatorname{tr}(A^3)\right),
\]
and therefore
\begin{equation}\label{eq:p4}
\operatorname{tr}(A^4)=\lambda\operatorname{tr}(A^3).
\end{equation}
Using \eqref{eq:p3} and \eqref{eq:p4},
\begin{align*}
|A(A-\lambda I)|^2
&=
\operatorname{tr}(A^4)
-2\lambda\operatorname{tr}(A^3)
+\lambda^2\operatorname{tr}(A^2)\\
&=0.
\end{align*}
Thus \eqref{eq:A-polynomial} holds, and \eqref{eq:nablaVA} gives
\begin{equation}\label{eq:nablaVA-zero}
\nabla_VA=0.
\end{equation}

Since $R$ is constant, \eqref{eq:gradR-general} gives
\begin{equation}\label{eq:AVzero}
A(V)=0.
\end{equation}
Differentiating \eqref{eq:AVzero} in the direction $X$ yields
\[
0=(\nabla_XA)(V)+A(\nabla_XV).
\]
By \eqref{eq:hess} and \eqref{eq:A-polynomial},
\[
A(\nabla_XV)=A(\lambda I-A)X=0.
\]
Hence
\begin{equation}\label{eq:nablaXAVzero}
(\nabla_XA)(V)=0.
\end{equation}

Set
\[
S=\nabla V=\lambda I-A.
\]
From the curvature convention,
\[
R(X,Y)V=(\nabla_XS)(Y)-(\nabla_YS)(X).
\]
Taking $Y=V$ and using \eqref{eq:nablaVA-zero} and
\eqref{eq:nablaXAVzero},
\begin{align*}
R(X,V)V
&=-(\nabla_XA)(V)+(\nabla_VA)(X)\\
&=0.
\end{align*}
Thus the soliton is radially flat. Since $R$ is constant, the rigidity
characterization of Petersen--Wylie
\cite[Theorem~1.2]{Petersen-Wylie2} implies that the soliton is rigid.
\end{proof}

\begin{remark}\label{rem:stationarity}
Every gradient Ricci soliton is self-similar under the Ricci flow: with $\tau(t)=1-2\lambda t$ and
$\sigma(t)=-(2\lambda)^{-1}\log\tau(t)$, the metrics
\[
g(t)=\tau(t)\,\Phi_{\sigma(t)}^*g,
\qquad \tau(t)>0,
\]
satisfy $\partial_tg=-2\operatorname{Ric}_{g(t)}$ and $g(0)=g$. Since the
covariant Ricci tensor is invariant under scaling,
$\operatorname{Ric}_{g(t)}=\Phi_{\sigma(t)}^*\operatorname{Ric}$: the
metric evolves by scalings and diffeomorphisms, the Ricci tensor by
diffeomorphisms alone. Differentiating
$\operatorname{Ric}_{g(t)}=\Phi_{\sigma(t)}^*\operatorname{Ric}$ at
$t=0$ and using Hamilton's equation
$\partial_t\operatorname{Ric}=\Delta_L\operatorname{Ric}$
\cite{Hamilton82} gives
\[
\mathcal L_{\nabla f}\operatorname{Ric}=\Delta_L\operatorname{Ric},
\]
where, in an orthonormal frame,
\[
\Delta_Lh=\Delta h+2\mathring R(h)-h(A\,\cdot,\cdot)-h(\cdot,A\,\cdot),
\qquad
\mathring R(h)(X,Y)=\sum_{k,l}h(e_k,e_l)\langle R(X,e_k)e_l,Y\rangle.
\]
The same identity follows from \eqref{eq:lie-nablaV} and
$\Delta_f\operatorname{Ric}=2\lambda\operatorname{Ric}
-2\mathring R(\operatorname{Ric})$ \cite[Proposition~2.1]{ELM}.
The hypothesis of Theorem~\ref{thm:riemannian-rigidity} thus says that, at each point
$x\in M$, the map $t\mapsto\operatorname{Ric}_{g(t)}(x)$ is monotone as
a family of quadratic forms on $T_xM$, and the theorem says that it is
then constant in time: $\operatorname{Ric}_{g(t)}=\operatorname{Ric}_g$
for every $t$. The flow is then linear,
\[
g(t)=g-2t\operatorname{Ric}.
\]
\end{remark}

\begin{proof}[Proof of Theorem~\ref{mainthm}]
By Theorem~\ref{thm:lie-ricci},
\[
\mathcal L_{\nabla f}\operatorname{Ric}=0.
\]
Theorem~\ref{thm:riemannian-rigidity} gives
\[
A(A-\lambda I)=0
\]
and the soliton is rigid. Hence the Ricci eigenvalues are $0$ and
$\lambda$.

Since $A$ commutes with $J$, both eigenspaces are $J$-invariant.
Therefore the rank of $A$ is even; write
\[
\operatorname{rank}A=2k,
\qquad 0\leq k\leq m.
\]
Then
\[
R=\operatorname{tr}A=2k\lambda.
\]
By rigidity, the universal cover splits isometrically as
\[
\widetilde M\cong N^{2k}\times\mathbb R^{2(m-k)},
\]
where $N$ is complete, simply connected, and Einstein with
\[
\operatorname{Ric}_{g_N}=\lambda g_N.
\]
The two factors are tangent to the $\lambda$- and $0$-eigenspaces of
$A$, respectively. These distributions are $J$-invariant and parallel,
so the parallel complex structure splits with the metric. Consequently,
$N$ is K\"ahler--Einstein and
\[
(\widetilde M,\widetilde g,\widetilde J)
\cong
N^{2k}\times\mathbb C^{m-k}
\]
holomorphically and isometrically. The general rigid description then
shows that $M$ is a quotient of this product by a discrete group of
holomorphic isometries acting freely and properly discontinuously.

Suppose now that $\lambda>0$. In the rigid description, $M$ is a flat
vector bundle over its zero section $\Sigma^{2k}$, and $T\Sigma$ is the
$\lambda$-eigenspace of $A$. Hence $\Sigma$ is a complete
K\"ahler--Einstein manifold satisfying
\[
\operatorname{Ric}_{\Sigma}=\lambda g_{\Sigma}>0.
\]
By Bonnet--Myers, $\Sigma$ is compact, and Kobayashi's theorem
\cite{Kobayashi1961} implies that $\Sigma$ is simply connected. A flat
vector bundle over a simply connected base has trivial holonomy and is
therefore trivial. Hence
\[
M^{2m}
\cong
\Sigma^{2k}\times\mathbb R^{2(m-k)}
\cong
\Sigma^{2k}\times\mathbb C^{m-k}
\]
holomorphically and isometrically. This proves the shrinking assertion
and completes the proof.
\end{proof}

\begin{remark}\label{rem:kahler-role}
Proposition~\ref{prop:kahler-lie} shows that, on a gradient
K\"ahler--Ricci soliton, the tensor
$\mathcal L_V\operatorname{Ric}$ is precisely the $J$-invariant part of
the Hessian of the scalar curvature. In particular, constant scalar
curvature forces $\mathcal L_V\operatorname{Ric}=0$.

The same mechanism appears along the Ricci flow. The metrics $g(t)$ of
Remark~\ref{rem:stationarity} are K\"ahler with respect to $J$, since
$\Phi_t$ preserves $J$. In holomorphic coordinates,
\[
\rho=-i\partial\bar\partial\log\det(g_{\alpha\bar\beta}),
\]
and
\[
\partial_t\log\det(g_{\alpha\bar\beta})=-R,
\]
so
\[
\partial_t\rho=i\partial\bar\partial R.
\]
Thus the Ricci form evolves only through the spatial variation of the
scalar curvature.

For comparison, on a general gradient Ricci soliton only the trace
identity \eqref{eq:lie-trace} remains. Suppose that $R$ is constant and
let
\[
T_V(X)=R(X,V)V
\]
be the radial Jacobi operator. Then
$\operatorname{Ric}(V,\cdot)=0$, and differentiating $A(V)=0$ together
with \eqref{eq:soliton-commutator} gives
\[
\nabla_VA=T_V-A(\lambda I-A).
\]
Using \eqref{eq:lie-nablaV}, we obtain
\[
(\mathcal L_V\operatorname{Ric})(X,Y)
=\bigl\langle\bigl(T_V+A(\lambda I-A)\bigr)X,Y\bigr\rangle.
\]
Hence, under constant scalar curvature,
\[
\mathcal L_V\operatorname{Ric}=0
\quad\Longleftrightarrow\quad
T_V=A^2-\lambda A.
\]
In the K\"ahler case the vanishing of
$\mathcal L_V\operatorname{Ric}$ follows directly from
Proposition~\ref{prop:kahler-lie}; in
Theorem~\ref{thm:riemannian-rigidity} it follows from the semidefinite
sign condition.
\end{remark}

\begin{corollary}\label{cor:psh}
Let $(M^{2m},g,f,J)$ be a complete nonsteady gradient K\"ahler--Ricci
soliton satisfying
\[
\operatorname{Ric}+\nabla^2f=\lambda g,
\qquad \lambda\neq0.
\]
If the scalar curvature is plurisubharmonic, or plurisuperharmonic,
then it is constant, and the conclusions of Theorem~\ref{mainthm} hold.
\end{corollary}

\begin{proof}
Since $\frac12\mathrm{d}\mathrm{d}^cR=i\partial\bar\partial R$,
\eqref{eq:kahler-lie} gives
\[
(\mathcal L_{\nabla f}\operatorname{Ric})(X,X)=i\partial\bar\partial R(X,JX).
\]
Thus $\mathcal L_{\nabla f}\operatorname{Ric}\geq0$ if $R$ is
plurisubharmonic, and $\mathcal L_{\nabla f}\operatorname{Ric}\leq0$ if
$R$ is plurisuperharmonic. Theorem~\ref{thm:riemannian-rigidity} shows
that $R$ is constant, and Theorem~\ref{mainthm} applies.
\end{proof}

\section{Rigidity under Ricci pinching}

We now prove Theorem~\ref{thm:pinching-main}. The proof is divided into
two parts. First we show directly that the pinching forces constant
scalar curvature and the algebraic relation $A^2=\lambda A$. We then
derive radial flatness from this algebraic relation.

\begin{lemma}\label{lem:pinching-constant}
Let $(M^n,g,f)$ be a complete gradient shrinking Ricci soliton with
\[
\operatorname{Ric}+\nabla^2f=\lambda g,
\qquad \lambda>0.
\]
If
\[
0\leq A\leq\lambda I,
\]
then $R$ is constant and
\begin{equation}\label{eq:A-projection-pinching}
A^2=\lambda A.
\end{equation}
\end{lemma}

\begin{proof}
By \eqref{eq:scalar-drift},
\begin{equation}\label{eq:scalar-pinching}
\Delta_fR
=
2\operatorname{tr}\bigl(A(\lambda I-A)\bigr).
\end{equation}
The endomorphism $A(\lambda I-A)$ is nonnegative and self-adjoint, so
\[
\Delta_fR\geq0.
\]
Moreover,
\[
0\leq R=\operatorname{tr}A\leq n\lambda.
\]

Let $\phi_r$ be a standard cutoff satisfying
\[
0\leq\phi_r\leq1,
\qquad
\phi_r\equiv1\ \text{on }B_r,
\qquad
\operatorname{supp}\phi_r\subset B_{2r},
\qquad
|\nabla\phi_r|\leq\frac{C}{r}.
\]
Since $R\geq0$ and $\Delta_fR\geq0$,
\[
0
\leq
\int_M\phi_r^2R\,\Delta_fR\,e^{-f}\,\mathrm{d}v.
\]
Weighted integration by parts gives
\begin{align*}
0
&\leq
-\int_M
\left\langle\nabla(\phi_r^2R),\nabla R\right\rangle e^{-f}\,\mathrm{d}v\\
&=
-\int_M\phi_r^2|\nabla R|^2e^{-f}\,\mathrm{d}v
-
2\int_M\phi_rR\langle\nabla\phi_r,\nabla R\rangle e^{-f}\,\mathrm{d}v.
\end{align*}
Therefore, by Young's inequality,
\[
\int_M\phi_r^2|\nabla R|^2e^{-f}\,\mathrm{d}v
\leq
4\int_MR^2|\nabla\phi_r|^2e^{-f}\,\mathrm{d}v.
\]
Using $R\leq n\lambda$, \eqref{eq:finite-weighted-volume}, and the
bound on $|\nabla\phi_r|$,
\[
\int_{B_r}|\nabla R|^2e^{-f}\,\mathrm{d}v
\leq
\frac{C}{r^2}\int_Me^{-f}\,\mathrm{d}v.
\]
Letting $r\to\infty$ gives
\[
\nabla R=0.
\]
Thus $R$ is constant. Returning to \eqref{eq:scalar-pinching},
\[
0=\operatorname{tr}\bigl(A(\lambda I-A)\bigr).
\]
Since $A(\lambda I-A)$ is nonnegative and self-adjoint,
\[
A(\lambda I-A)=0,
\]
which is equivalent to \eqref{eq:A-projection-pinching}.
\end{proof}

The following lemma isolates the partial Codazzi structure used to prove
radial flatness.

\begin{lemma}\label{lem:partial-codazzi}
Let $(M^n,g,f)$ be a connected gradient shrinking Ricci soliton with
$\lambda>0$. Assume
\begin{equation}\label{eq:A-projection}
A^2=\lambda A.
\end{equation}
Let
\[
E_0=\ker A,
\qquad
E_\lambda=\ker(A-\lambda I).
\]
Then
\begin{equation}\label{eq:partial-codazzi}
(\nabla_XA)Y=(\nabla_YA)X
\end{equation}
whenever $X,Y$ both belong to $E_0$ or both belong to $E_\lambda$.
\end{lemma}

\begin{proof}
Since $A/\lambda$ is a smooth self-adjoint projection, its rank is
constant on the connected manifold $M$. Hence
\[
R=\operatorname{tr}A
\]
is constant and, by \eqref{eq:divA},
\begin{equation}\label{eq:divA-zero}
\operatorname{div}A=0.
\end{equation}

Fix $p\in M$. Choose an orthonormal eigenbasis
$\{e_i\}_{i=1}^n$ of $A$ at $p$ and extend it to an orthonormal frame
normal at $p$. Write
\[
Ae_i=\mu_i e_i,
\qquad
\mu_i\in\{0,\lambda\},
\]
and set
\[
a_{ijk}
=
\left\langle(\nabla_{e_i}A)e_j,e_k\right\rangle.
\]
Since $A$ is self-adjoint,
\begin{equation}\label{eq:a-last-sym}
a_{ijk}=a_{ikj}.
\end{equation}

Differentiating \eqref{eq:A-projection} in the direction $e_i$ gives
\[
(\nabla_{e_i}A)A
+
A(\nabla_{e_i}A)
=
\lambda\nabla_{e_i}A.
\]
Taking the $(e_j,e_k)$ component,
\begin{equation}\label{eq:selection}
(\mu_j+\mu_k-\lambda)a_{ijk}=0.
\end{equation}
Consequently,
\begin{equation}\label{eq:same-eigen-zero}
a_{ijk}=0
\qquad\text{if}\qquad
\mu_j=\mu_k.
\end{equation}

Because of \eqref{eq:divA-zero} and \eqref{eq:A-projection},
\[
0
=
\operatorname{div}(A^2)
=
\sum_i(\nabla_{e_i}A)(Ae_i)+A(\operatorname{div}A).
\]
Thus the vector field
\[
W=\sum_i(\nabla_{e_i}A)(Ae_i)
\]
vanishes. Taking its divergence at $p$ gives
\begin{equation}\label{eq:divW}
0
=
\sum_{i,j}
\left\langle(\nabla^2_{e_j,e_i}A)e_j,Ae_i\right\rangle
+
\sum_{i,j,k}a_{ijk}a_{jik},
\end{equation}
where
\[
\nabla^2_{X,Y}A
=
\nabla_X(\nabla_YA)-\nabla_{\nabla_XY}A.
\]
Differentiating \eqref{eq:divA-zero} and contracting with $Ae_i$ gives
\begin{equation}\label{eq:divdivA}
0
=
\sum_{i,j}
\left\langle(\nabla^2_{e_i,e_j}A)e_j,Ae_i\right\rangle.
\end{equation}
Subtracting \eqref{eq:divW} from \eqref{eq:divdivA}, and using the
Ricci commutation formula for endomorphisms,
\[
(\nabla^2_{X,Y}-\nabla^2_{Y,X})A=[R(X,Y),A],
\]
we obtain
\begin{align}
\sum_{i,j,k}a_{ijk}a_{jik}
&=
\sum_{i,j}
\left\langle[R(e_i,e_j),A]e_j,Ae_i\right\rangle
\nonumber\\
&=
\sum_{i,j}\mu_i(\mu_j-\mu_i)K_{ij}
\nonumber\\
&=
Q(A)-\operatorname{tr}(A^3).
\label{eq:a-curvature}
\end{align}
Since \eqref{eq:A-projection} gives
\[
\operatorname{tr}(A^3)
=
\lambda\operatorname{tr}(A^2)
=
\lambda|\operatorname{Ric}|^2,
\]
and $|\operatorname{Ric}|^2=\lambda R$ is constant, the Bochner formula
\eqref{eq:bochner-ricci} yields
\[
|\nabla\operatorname{Ric}|^2
=
2\bigl(Q(A)-\lambda|\operatorname{Ric}|^2\bigr).
\]
Together with \eqref{eq:a-curvature},
\begin{equation}\label{eq:key-norm}
|\nabla\operatorname{Ric}|^2
=
2\sum_{i,j,k}a_{ijk}a_{jik}.
\end{equation}

If $a_{ijk}a_{jik}\neq0$, then \eqref{eq:selection} implies
\[
\mu_j\neq\mu_k,
\qquad
\mu_i\neq\mu_k.
\]
Since the only eigenvalues are $0$ and $\lambda$, we must have
$\mu_i=\mu_j$. Hence \eqref{eq:key-norm} becomes
\begin{equation}\label{eq:key-norm-restricted}
\frac12|\nabla\operatorname{Ric}|^2
=
\sum_{\mu_i=\mu_j}a_{ijk}a_{jik}.
\end{equation}

By \eqref{eq:same-eigen-zero}, every nonzero $a_{ijk}$ satisfies
$\mu_j\neq\mu_k$. Using \eqref{eq:a-last-sym}, the terms
$a_{ijk}$ and $a_{ikj}$ occur in equal pairs, and exactly one member of
each pair satisfies $\mu_i=\mu_j$. Therefore
\begin{equation}\label{eq:half-norm}
\sum_{\mu_i=\mu_j}a_{ijk}^2
=
\frac12|\nabla\operatorname{Ric}|^2.
\end{equation}
Applying Cauchy--Schwarz to \eqref{eq:key-norm-restricted},
\begin{align*}
\frac12|\nabla\operatorname{Ric}|^2
&=
\sum_{\mu_i=\mu_j}a_{ijk}a_{jik}\\
&\leq
\frac12\sum_{\mu_i=\mu_j}
\bigl(a_{ijk}^2+a_{jik}^2\bigr)\\
&=
\sum_{\mu_i=\mu_j}a_{ijk}^2\\
&=
\frac12|\nabla\operatorname{Ric}|^2.
\end{align*}
Equality holds throughout, so
\[
a_{ijk}=a_{jik}
\qquad\text{whenever}\qquad
\mu_i=\mu_j.
\]
This is exactly \eqref{eq:partial-codazzi}.
\end{proof}

\begin{remark}\label{rem:codazzi-connection}
The symmetry in Lemma~\ref{lem:partial-codazzi} is weaker than the full
Codazzi condition, since it is asserted only when the two input vectors
belong to the same eigendistribution of $A$. Nevertheless, it is closely
related to the rigidity mechanism for Codazzi operators studied by
Cheng, Lipa and Zhou \cite{Cheng-Lipa-Zhou}, where scalar curvature and
trace invariants are used to control the eigenvalue structure of a
Codazzi symmetric tensor. Here the combination of $A^2=\lambda A$ with
the Ricci Bochner identity produces exactly the partial symmetry needed
for radial flatness.
\end{remark}

\begin{proposition}\label{prop:projection-radial}
Let $(M^n,g,f)$ be a complete connected gradient shrinking Ricci soliton
with $\lambda>0$. If
\[
A^2=\lambda A,
\]
then
\[
R(X,\nabla f)\nabla f=0
\]
for every vector field $X$. Hence the soliton is radially flat and,
by \cite[Theorem~1.2]{Petersen-Wylie2}, rigid.
\end{proposition}

\begin{proof}
As observed in the proof of Lemma~\ref{lem:partial-codazzi},
$R=\operatorname{tr}A$ is constant. Hence
\[
A(V)=0
\]
by \eqref{eq:gradR-general}; in particular, $V\in E_0$.

Differentiating $A(V)=0$ gives
\[
0=(\nabla_XA)V+A(\nabla_XV).
\]
Using \eqref{eq:hess} and $A(\lambda I-A)=0$,
\begin{equation}\label{eq:nablaA-V-zero}
(\nabla_XA)V=0
\end{equation}
for every $X$.

If $F\in E_0$, then both $F$ and $V$ belong to the same eigenspace.
Lemma~\ref{lem:partial-codazzi} and \eqref{eq:soliton-commutator} give
\[
R(F,V)V=0.
\]
Thus the Jacobi operator
\[
T_V(X)=R(X,V)V
\]
vanishes on $E_0$.

On the other hand, taking $Y=V$ in \eqref{eq:soliton-commutator} and
using \eqref{eq:nablaA-V-zero},
\begin{equation}\label{eq:TV-nablaVA}
T_V(X)=(\nabla_VA)X.
\end{equation}
Differentiating $A^2=\lambda A$ in the direction $V$ shows that
$\nabla_VA$ has only mixed components between $E_0$ and $E_\lambda$.
Hence, if $E\in E_\lambda$,
\[
T_V(E)\in E_0.
\]
The Jacobi operator $T_V$ is self-adjoint. Therefore, for $F\in E_0$,
\[
\langle T_VE,F\rangle
=
\langle E,T_VF\rangle
=
0.
\]
Since $T_VE\in E_0$, this forces $T_VE=0$. Thus $T_V=0$ on both
$E_0$ and $E_\lambda$, and consequently
\[
R(X,V)V=0
\]
for every $X$. At critical points of $f$ the same conclusion is
automatic. Hence the soliton is radially flat.

Since the scalar curvature is constant, the rigidity characterization
of Petersen--Wylie \cite[Theorem~1.2]{Petersen-Wylie2} implies rigidity.
\end{proof}

\begin{proof}[Proof of Theorem~\ref{thm:pinching-main}]
Lemma~\ref{lem:pinching-constant} gives
\[
R=\text{constant}
\qquad\text{and}\qquad
A^2=\lambda A.
\]
Proposition~\ref{prop:projection-radial} then gives
\[
R(X,\nabla f)\nabla f=0
\]
for every $X$, and hence the soliton is rigid.
\end{proof}

\begin{remark}\label{rem:PW-independent}
Theorem~\ref{thm:pinching-main} provides a direct proof under the
pinching
\[
0\leq\operatorname{Ric}\leq\lambda g.
\]
The scalar curvature is first shown to be constant by a weighted cutoff
argument. The relation $A^2=\lambda A$ then determines the block
structure of $\nabla A$, while Lemma~\ref{lem:partial-codazzi} supplies
the additional symmetry needed to obtain radial flatness. Thus the
argument is independent of Proposition~1.3(2) of Petersen--Wylie
\cite{Petersen-Wylie2}.
\end{remark}

\begin{remark}\label{rem:Cheng-Zhou}
Proposition~\ref{prop:projection-radial} also gives an alternative proof
for the final rigidity step in the four-dimensional theorem of
Cheng--Zhou \cite{Cheng-Zhou}. Precisely, in their normalized setting,
\[
\operatorname{Ric}+\nabla^2f=\frac12g,
\qquad
R=1.
\]
Once their argument yields the rank-two spectrum
\[
0,\quad0,\quad\frac12,\quad\frac12,
\]
one has
\[
A^2=\frac12A.
\]
Proposition~\ref{prop:projection-radial} then gives radial flatness and
rigidity directly.
\end{remark}

\section*{Acknowledgments}
M. Hor\'acio thanks the Departamento de Matem\'atica of the
Universidade de Bras\'ilia, where part of this work was carried out, and
Jo\~ao Paulo dos Santos for his guidance and support.

\bibliographystyle{amsplain}

\end{document}